\documentclass[a4paper,10pt, oneside]{article}
\usepackage{amsmath, amssymb, amsthm, bm, mathtools}
\usepackage[titletoc, title]{appendix}
\usepackage[pdftex, pdfpagelabels, bookmarks, hyperindex, hyperfigures,
colorlinks, pagebackref]{hyperref}
\usepackage[capitalise, nosort]{cleveref}

\crefname{equation}{}{}
\crefname{lem}{Lemma}{Lemmas}
\crefname{thm}{Theorem}{Theorems}

\DeclareMathOperator*{\supp}{supp}

\newcommand{\dual}[1]{\left\langle #1 \right\rangle}
\newcommand{\D}{\mathrm{D}}
\newcommand{\E}{\mathrm{e}}

\newcommand{\nm}[1]{\left\Vert #1 \right\Vert}
\newcommand{\snm}[1]{\left\vert #1 \right\vert}
\newcommand{\snmii}[1]
{
  \left\vert\kern-0.25ex
  \left\vert\kern-0.25ex
  \left\vert
  #1
  \right\vert\kern-0.25ex
  \right\vert\kern-0.25ex
  \right\vert
}

\newtheorem{Def}{Definition}[section]
\newtheorem{lem}{Lemma}[section]
\newtheorem{rem}{Remark}[section]
\newtheorem{thm}{Theorem}[section]

\numberwithin{equation}{section}

\begin{document}
\title
{
  \Large\bf Regularity of solutions to time fractional diffusion equations
   \thanks
  {
    This work was supported by  Major Research Plan of National
    Natural Science Foundation of China (91430105).
  }
}
\author{
  Binjie Li \thanks{Email: libinjie@scu.edu.cn},
  Xiaoping Xie \thanks{Corresponding author. Email: xpxie@scu.edu.cn} \\
  {School of Mathematics, Sichuan University, Chengdu 610064, China}
}
\date{}
\maketitle

\begin{abstract}
  We derive some regularity estimates of the solution to a time fractional
  diffusion equation, that are useful for numerical analysis, and partially
  unravel the singularity structure of the solution with respect to the time
  variable. \vskip 0.2cm\noindent {\bf Keywords:} time fractional diffusion,
  regularity.
\end{abstract}

\section{Introduction}
This paper considers the following time fractional diffusion problem:
\begin{equation}
  \label{eq:model}
  \left\{
    \begin{array}{rl}
      \partial_t^\alpha (u-u_0) - \Delta u = f \phantom{0}
      & \text{ in $ \Omega \times (0,T) $, } \\
      u = 0 \phantom{0} & \text{ on $ \partial\Omega \times [0,T] $, } \\
      u = u_0 & \text{ on $ \Omega \times \{ 0 \} $, }\\
    \end{array}
  \right.
\end{equation}
where $ \Omega \subset \mathbb R^d $ is a bounded domain with $ C^2 $ boundary,
$ 0.5 < \alpha < 1 $, $ u_0 \in H_0^1(\Omega) $, and $ f \in L^2(Q_T) $ with $
Q_T := \Omega \times (0,T) $. Above $ \partial_t^\alpha: L^1(Q_T) \to \mathcal
D'(Q_T) $, the Riemann-Liouville fractional differential operator, is defined by
$ \partial_t^\alpha := \partial_t I_{0+}^{1-\alpha} $, where $ \partial_t $
denotes the generalized differential operator with respect to the time variable
$ t $, and $ I_{0+}^{1-\alpha}: L^1(Q_T) \to L^1(Q_T) $ is given by
\[
  (I_{0+}^{1-\alpha} v)(x,t) := \frac1{\Gamma(1-\alpha)}
  \int_0^t (t-s)^{-\alpha} v(x,s) \, \mathrm{d}s,
  \quad (x,t) \in Q_T,
\]
for all $ v \in L^1(Q_T) $, with $ \Gamma(\cdot) $ denoting the standard Gamma
function.


The main goal of this paper is to investigate the regularity properties of the
solution to problem \cref{eq:model}. Let us first introduce some recent related
works. In \cite{Eidelman;2004} Eidelman and Kochubei constructed and
investigated fundamental solutions for Cauchy time-fractional diffusion
problems. For the generalized time-fractional diffusion equation, Luchko
\cite{Luchko;2009, Luchko;2010} established a maximum principle, and used this
maximum principle to prove the uniqueness of the generalized solution; moreover,
Luchko discussed the existence of the generalized solution of problem
\cref{eq:model} with $ f = 0 $. Furthermore, Luchko \cite{Luchko;2012} extended
the maximum principle to the fractional diffusion equation, and discussed the
properties of the solution to a one-dimensional time-fractional diffusion
equation. Under the basic condition that $ f = 0 $ or $ u_0 = 0 $, Sakamoto and
Yamamoto \cite{Sakamoto;2011} discussed the uniqueness and regularity of the
weak solution to problem \cref{eq:model} for $ 0 < \alpha < 1 $. Zacher
\cite{Zacher;2013} proposed a De Giorgi-Nash theorem for time fractional
diffusion equations. Assuming the force function $ f $ to be weighted H\"older
continuous, Mu et al.\ \cite{Mu;2016} proved the unique existence of solutions
to three types of time-fractional diffusion equations, and derived some new
regularity estimates. For the abstract time-fractional evolution equations,
Zhang and Liu \cite{Zhang;2012} established sufficient conditions for the
existence of mild solutions for fractional evolution differential equations by
using a new fixed point theorem. Wang et al.\ \cite{Wang;2012} obtained the
existence and uniqueness of mild solutions and classical solutions to abstract
linear and semilinear fractional Cauchy problems with almost sectorial
operators. For more works, we refer the reader to
\cite{Mainardi;1994,Mainardi;1995,Mainardi;1996,El-sayed;1996,Agarawal;2002,
Gejji;2006,Fan;2014} and the references therein.

The motivation of paper is to provide regularity estimates for numerical
analysis, and, to the best of our knowledge, for problem \cref{eq:model} there
is no available regularity result for numerical analysis so far. Moreover,
because of the nonlocal property of the operator $ \partial_t^\alpha $, both the
storage and computing costs for a numerical approximate of a time fractional
diffusion problem are much more expensive than that of a corresponding standard
diffusion problem. Therefore, designing high accuracy algorithms for time
fractional diffusion problems is of great practical value
\cite{Li;Xu;2009,Zheng;2015}. But whether a high accuracy algorithm works or not
depends mainly on the regularity of the solution with respect to the time
variable $ t $. Unfortunately, it is well known that even $ u_0 $ and $ f $ are
regular enough, the solution to \cref{eq:model} has singularity in time. So it is
natural to investigate the singularity structure of the solution in time, which
not only is of theoretical value, but also can provide insight into developing
efficient numerical algorithms.


In this paper, we employ the Galerkin method to investigate the regularity
properties of the weak solution to problem \cref{eq:model}. Compared to the work
aforementioned, our regularity estimates are more applicable to numerical
analysis. Furthermore, our regularity estimates demonstrate that, by subtracting
some particular forms of singularity functions, we can improve the regularity of
the solution with respect to the time variable $ t $, which partially unravel
the singularity structure of the solution in time.


The rest of this paper is organized as follows. In \cref{sec:notation} we
introduce some Sobolev spaces, the fractional integration and derivative
operators, and some fundamental properties of these operators. In \cref{sec:ode}
we investigate the regularity properties of the solution of a fractional
ordinary equation. Finally, in \cref{sec:main} we use the results developed in
the previous sections to discuss the regularity of the solution to problem
\cref{eq:model}.

\section{Preliminaries}
\label{sec:notation}
We start by introducing some Sobolev spaces. Let $ 0 \leqslant \beta < \infty $.
Define \cite{Tartar;2007}
\[
  H^\beta(\mathbb R) := \left\{
    v \in L^2(\mathbb R) \middle| \
    \left( 1+\snm{\cdot}^2 \right) ^{\beta/2}
    \mathcal F v(\cdot) \in L^2(\mathbb R)
  \right\},
\]
and endow this space with the following norm:
\begin{align*}
  \nm{v}_{ H^\beta(\mathbb R) } :=
  \nm{v}_{ L^2(\mathbb R) } +
  \snm{v}_{ H^\beta(\mathbb R) } \quad
  \text{ for all $ v \in H^\beta(\mathbb R) $, }
\end{align*}
where
\[
  \snm{v}_{H^\beta(\mathbb R)} :=
  \left(
    \int_{\mathbb R} \snm{\xi}^\beta
    \snm{ \mathcal Fv(\xi) }^2
    \, \mathrm{d}\xi
  \right)^\frac12.
\]
Above and throughout, $ \mathcal F: \mathcal S'(\mathbb R) \to \mathcal
S'(\mathbb R) $ denotes the well-known Fourier transform operator, where $
\mathcal S'(\mathbb R) $, the dual space of $ \mathcal S(\mathbb R) $, is the
space of tempered distributions. For $ -\infty \leqslant a < b \leqslant \infty
$, define
\[
  H^\beta(a,b) := \left\{
    v|_{(a,b)} \middle| \
    v \in H^\beta(\mathbb R)
  \right\},
\]
and equip this space with the following norm:
\[
  \nm{v}_{H^\beta(a,b)} :=
  \inf_{
    \substack{
      \widetilde v \in H^\beta(\mathbb R) \\
      \widetilde v |_{(a,b)} = v
    }
  } \nm{ \widetilde v }_{ H^\beta(\mathbb R) }
  \quad \text{ for all $ v \in H^\beta(a,b) $. }
\]
In addition, we use $ H_0^\beta(a,b) $ to denote the closure of $ \mathcal
D(a,b) $ in $ H^\beta(a,b) $, where $ \mathcal D(a,b) $ denotes the set of $
C^\infty $ functions with compact support in $ (a, b) $.
\begin{rem}
  It is well known that if $ 0 < \beta < 0.5 $, then $ H^\beta(a,b) $ coincides
  with $ H_0^\beta(a,b) $. Also, there exists another equivalent definition of
  the space $ H^\beta(0,1) $ for $ 0 < \beta < 1 $:
  \[
    H^\beta(0,T) := \left\{
      v \in L^2(0,T) \middle| \
      \int_0^1 \int_0^1
      \frac{
        \snm{v(s) - v(t)}^2
      }{
        \snm{s-t}^{1+2\beta}
      }
      \, \mathrm{d}s \, \mathrm{d}t < \infty
    \right\}.
  \]
  By this definition, a routine computation yields that if $ 0 < \beta < 0.25 $,
  then $ v \in H^\beta(0,1) $ with $ v $ being given by
  \[
    v(t) := t^{-\beta}, \quad 0 < t < 1.
  \]
  This result will be used implicitly in the proof of \cref{thm:ode-2}.
\end{rem}

Let $ X $ be a separable Hilbert space with inner product $ (\cdot,\cdot)_X $,
and an orthonormal basis $ \{ e_k | \ k \in \mathbb N \} $. For $ -\infty < a <
b < \infty $ and $ 0 \leqslant \beta < \infty $, define
\[
  H^\beta(a,b; X) := \left\{
    v:\ (a,b) \to X \middle| \
    \sum_{k=0}^\infty \nm{(v,e_k)_X}_{H^\beta(a,b)}^2
    < \infty
  \right\},
\]
and equip this space with the following norm:
\[
  \nm{v}_{H^\beta(0,T; X)} := \left(
    \sum_{k=0}^\infty \nm{(v,e_k)_X}_{H^\beta(0,T)}^2
  \right)^\frac12 \quad
  \text{ for all $ v \in H^\beta(a,b; X) $. }
\]
It is easy to verify that $ H^\beta(a,b; X) $ is a Banach space, and, in
particular, we shall also use $ L^2(a, b; X) $ to denote the space $ H^0(a,b; X)
$.
\begin{rem}
  It is evident that the spaces $ L^2(a,b; X) $ and $ H^1(a,b; X) $ defined
  above coincide respectively with the corresponding standard $ X $-valued
  Sobolev spaces \cite{Evans}, with the same norms. Using the $ K $-method
  \cite{Tartar;2007}, we see that, for $ 0 < \beta < 1 $, the space $
  H^\beta(a,b; X) $ coincides with the following interpolation space
  \[
    \big( L^2(a,b; X),\ H^1(a,b; X) \big)_{\beta,2},
  \]
  with equivalent norms. Thus, the space $ H^\beta(a,b; X) $, $ 0 \leqslant
  \beta \leqslant 1 $, is independent of the choice of orthonormal basis $ \{
  e_k |\ k \in \mathbb N \} $ of $ X $. The case of $ \beta > 1 $ is analogous.
\end{rem}

Then, let us introduce the Riemann-Liouville fractional integration and
derivative operators as follows \cite{Samko;1993,Podlubny;1999}.
\begin{Def}
  Let $ \beta > 0 $. Define $ I_{+}^\beta: V_{+} \to V_{+} $ by
  \[
    I_{+}^\beta v(x) := \frac1{ \Gamma(\beta) }
    \int_{-\infty}^x (x-t)^{\beta-1} v(t) \, \mathrm{d}t,
    \quad -\infty < x < \infty,
  \]
  for all $ v \in V_{+} $, and define $ I_{-}^\beta: V_{-} \to V_{-} $ by
  \[
    I_{-}^\beta v(x) := \frac1{ \Gamma(\beta) }
    \int_x^\infty (t-x)^{\beta-1} v(t) \, \mathrm{d}t,
    \quad -\infty < x < \infty,
  \]
  for all $ v \in V_{-} $. Here $ \Gamma(\cdot) $ denotes the standard Gamma
  function, and
  \begin{align*}
    V_{+} &:= \left\{
      v \in L_{\mathrm{loc}}^1(\mathbb R) \middle| \
      \supp v \subset [a, \infty)
      \text{ for some } a \in \mathbb R
    \right\}, \\
    V_{-} &:= \left\{
      v \in L_{ \mathrm{loc} }^1(\mathbb R) \middle| \
      \supp v \subset (-\infty, a]
      \text{ for some } a \in \mathbb R
    \right\},
  \end{align*}
  with $ \supp v $ denoting the support of $ v $ and $
  L_{\mathrm{loc}}^1(\mathbb R) $ denoting the set of all locally integrable
  functions defined on $ \mathbb R $. In particular, for $ a \in \mathbb R $, we
  use $ I_{a+}^\beta $ to denote the restriction of $ I_{+}^\beta $ to
  \[
    \left\{
      v \in V_{+} \middle| \ \supp v \subset [a, \infty)
    \right\},
  \]
  and use $ I_{a-}^\beta $ to denote the restriction of $ I_{-}^\beta $ to
  \[
    \left\{
      v \in V_{-} \middle| \ \supp v \subset (-\infty, a]
    \right\}.
  \]
\end{Def}

\begin{Def}
  Let $ 0 < \beta < 1 $. Define $ \D_{+}^\beta: \ V_{+} \to \mathcal D'(\mathbb
  R) $ and $ \D_{-}^\beta: \ V_{-} \to \mathcal D'(\mathbb R) $, respectively,
  by
  \[
    \D_{+}^\beta := \D I_{+}^{1-\beta}, \quad
    \D_{-}^\beta := -\D I_{-}^{1-\beta},
  \]
  where $ \D: \mathcal D'(\mathbb R) \to \mathcal D'(\mathbb R) $ is the
  standard generalized differential operator. In particular, for $ a \in
  \mathbb R $, we use $ \D_{a+}^\beta $ to denote the restriction of $
  \D_{+}^\beta $ to
  \[
    \left\{
      v \in V_{+} \middle| \
      \supp v \subset [a, \infty)
    \right\},
  \]
  and use $ \D_{a-}^\beta $ to denote the restriction of $ \D_{-}^\beta $ to
  \[
    \left\{
      v \in V_{-} \middle| \
      \supp v \subset (-\infty, a]
    \right\}.
  \]
\end{Def}
\begin{rem}
  \label{rem:frac_deri}
  Let $ 0 < \beta < 1 $ and $ v \in V_{+} $. By the definition of $
  I_{+}^{1-\beta} $, we have $ I_{+}^{1-\beta} v = h*v $, where
  \[
    h(t) := \frac1{ \Gamma(1-\beta) }
    \begin{cases}
      t^{-\beta}, & \text{ if $ t > 0 $, } \\
      0, & \text{ otherwise. }
    \end{cases}
  \]
  If the support of $ v $ is compact, then \cite[Theorem 6, pp.\
  160-161]{Yosida;1980} implies $ \mathcal F(h*v) = \mathcal F h \mathcal F v $,
  so that, using the fact that
  \[
    \mathcal Fh(\xi) = (\mathrm{i}\xi)^{\beta-1}, \quad -\infty < \xi < \infty,
  \]
  which can be derived by a similar analysis as in \cref{sec:app}, we obtain
  \begin{align*}
    \mathcal F ( I_{+}^{1-\beta} v )(\xi) &=
    ( \mathrm{i}\xi )^{\beta-1} \mathcal F v(\xi), \\
    \mathcal F ( \D_{+}^\beta v )(\xi) &=
    ( \mathrm{i}\xi )^\beta \mathcal F v(\xi),
  \end{align*}
  for all $ -\infty < \xi < \infty $. This provides an approach to prove
  \cref{lem:F} below.
\end{rem}
\begin{rem}
  Note that these fractional integration and derivative operators introduced
  above act on functions defined on $ \mathbb R $. For the sake of rigorousness,
  we make the convention that, when applying one of these operators to a
  function $ v $ defined in some interval $ (a,b) $, we shall implicitly extend
  $ v $ to $ \mathbb R \setminus (a,b) $ by zero.
\end{rem}

In the remainder of this section, we present some fundamental properties of
these fractional integration and derivative operators.
\begin{lem}
  \label{lem:basic-1}
  If $ \beta, \gamma > 0 $, then
  \[
    I_{+}^{ \beta + \gamma } = I_{+}^\beta I_{+}^\gamma, \quad
    I_{-}^{ \beta + \gamma } = I_{-}^\beta I_{-}^\gamma.
  \]
\end{lem}
\begin{lem}
  \label{lem:basic-2}
  Let $ -\infty < a < b < \infty $ and $ \beta > 0 $. If $ u $, $ v \in L^2(a,b)
  $, then
  \[
    \left( I_{a+}^\beta u, v \right)_{L^2(a,b)} =
    \left( u, I_{b-}^\beta v \right)_{L^2(a,b)}.
  \]
  If $ v \in L^p(a,b) $ with $ 1 \leqslant p \leqslant \infty $, then
  \begin{align*}
    \nm{I_{a+}^\beta v}_{L^p(a,b)} &\leqslant C \nm{v}_{L^p(a,b)}, \\
    \nm{I_{b-}^\beta v}_{L^p(a,b)} &\leqslant C \nm{v}_{L^p(a,b)},
  \end{align*}
  where $ C $ is a positive constant that only depends on $ a $, $ b $,
  $ \beta $ and $ p $.
\end{lem}

\begin{lem}
  \label{lem:F}
  Let $ 0 < \beta < 1 $ and $ v \in L^1(\mathbb R) $ with compact support. Then
  \begin{align*}
    ( \mathcal F I_{+}^\beta v )(\xi) &=
    ( \mathrm{i}\xi )^{-\beta} \mathcal Fv(\xi), \\
    ( \mathcal F \D_{+}^\beta v )(\xi) &=
    ( \mathrm{i}\xi )^\beta \mathcal F v(\xi), \\
    ( \mathcal F I_{-}^\beta v )(\xi) &=
    ( -\mathrm{i}\xi )^{-\beta} \mathcal Fv(\xi), \\
    ( \mathcal F \D_{-}^\beta v )(\xi) &=
    ( -\mathrm{i}\xi )^\beta \mathcal F v(\xi),
  \end{align*}
  for all $ \xi \in \mathbb R $.
\end{lem}
\noindent For the proofs of \cref{lem:basic-1,lem:basic-2}, we refer the reader
to \cite{Samko;1993}, and, since these two lemmas are frequently used in this
paper, we shall use them without notice for convenience. For \cref{lem:F}, a
proof was also provided in \cite[Theorem 7.1, pp.\ 138]{Samko;1993}; however,
this proof was not rigorous. Although one can prove \cref{lem:F} by using the
convolution theorem for distributions, as described in \cref{rem:frac_deri}, we
provide an elementary proof in \cref{sec:app}.
\begin{rem}
  Let $ 0 < \beta < 1 $. If $ v \in H^\beta(\mathbb R) $ with compact support,
  then using \cref{lem:F} and the famous Fourier-Plancherel formula gives
  \[
    \nm{\D_{+}^\beta v}_{L^2(\mathbb R)} =
    \snm{v}_{H^\beta(\mathbb R)} .
  \]
  This is a remarkable and very useful property of fractional derivative
  operators. Also, when applying \cref{lem:F}, we should be cautious: for $ v
  \in H^\beta(\mathbb R) $, in our setting $ \mathcal F(\D_{+}^\beta v) $ may
  not make any sense since the support of $ v $ is not necessarily compact.
\end{rem}

Below we make the following conventions: by $ x \lesssim y $ we mean that there
exists a positive constant $ C $ that only depends on $ \alpha $, $ T $ or $
\Omega $, unless otherwise specified, such that $ x \leqslant C y $ (the value
of $ C $ may differ at each occurrence); by $ x \sim y $ we mean that $ x
\lesssim y \lesssim x $.

\begin{lem}
  \label{lem:regu_alpha}
  Let $ 0 < \beta < 1 $. If $ h \in L^2(0,T) $, then
  \begin{align}
    \nm{ I_{0+}^\beta h }_{ H^\beta(0,T) } &
    \lesssim \nm{h}_{L^2(0,T)},
    \label{eq:regu_alpha-1} \\
    \nm{ I_{T-}^\beta h }_{ H^\beta(0,T) } &
    \lesssim \nm{h}_{L^2(0,T)}.
    \label{eq:regu_alpha-2}
  \end{align}
  If $ h \in H_0^{1-\beta}(0,T) $, then
  \begin{align}
    \nm{ I_{0+}^\beta h }_{H^1(0,T)} &
    \lesssim \nm{h}_{ H^{1-\beta}(0,T) },
    \label{eq:regu_alpha-3}
  \end{align}
  and if $ \beta \neq 0.5 $ in addition, then $ I_{0+}^\beta h(0) = 0 $. The
  three implicit constants in the above three estimates only depend on $ \beta $
  and $ T $.
\end{lem}
\begin{proof}
  Let us first prove \cref{eq:regu_alpha-1,eq:regu_alpha-2}, and since the proof
  of \cref{eq:regu_alpha-2} is parallel to that of \cref{eq:regu_alpha-1}, here
  we only prove the former. Assuming $ 0 < \beta < \frac12 $, a simple
  calculation gives
  \[
    \nm{I_{0+}^\beta h}_{ L^2(\mathbb R) } \lesssim \nm{h}_{L^2(0,T)},
  \]
  so that, \cref{lem:F} and the Fourier-Plancherel formula imply
  \[
    \nm{I_{0+}^\beta h}_{ H^\beta(\mathbb R) } \lesssim
    \nm{h}_{L^2(0,T)},
  \]
  which indicates the desired estimate \cref{eq:regu_alpha-1}. So let us assume
  $ \frac12 \leqslant \beta < 1 $. Note that we have already proved
  \begin{align*}
    \nm{ I_{0+}^\frac \beta2 v }_{ H^\frac \beta2(0,T) } \lesssim
    \nm{v}_{L^2(0,T)} \quad
    \text{ for all $ v \in L^2(0,T) $. }
  \end{align*}
   Also note that for any $ v \in H_0^1(0,T) $, since
  \[
    \D I_{0+}^\frac \beta2 v =
    I_{0+}^\frac \beta2 \D v \quad
    \text{ in $ (0,T) $, }
  \]
  we have
  \[
    \nm{ I_{0+}^\frac \beta2 v }_{ H^{1+\frac \beta2}(0,T) }
    \lesssim \nm{v}_{H_0^1(0,T)}.
  \]
  Consequently, the standard properties of interpolation spaces implies
  \[
    \nm{ I_{0+}^\frac \beta2 v }_{ H^\beta(0,T) } \lesssim
    \nm{v}_{ H^\frac \beta2(0,T) } \quad
    \text{ for all $ v \in H^\frac \beta2(0,T) $. }
  \]
  Noting the fact that $ I_{0+}^\beta h = I_{0+}^\frac \beta2 I_{0+}^\frac
  \beta2 h $ and
  \[
    \nm{ I_{0+}^\frac \beta2 h }_{ H^\frac \beta2(0,T) } \lesssim
    \nm{h}_{L^2(0,T)},
  \]
  we immediately obtain \cref{eq:regu_alpha-1}.

  Then let us prove \cref{eq:regu_alpha-3}. Given $ v \in H_0^1(0,T) $, since
  \[
    \D I_{0+}^\beta v = I_{0+}^\beta \D v \quad \text{ in $ (0,T) $, }
  \]
  by \cref{eq:regu_alpha-1} it follows
  \[
    \nm{ I_{0+}^\beta v }_{ H^{1+\beta}(0,T) } \lesssim \nm{v}_{H^1(0,T)}.
  \]
  Therefore, using \cref{eq:regu_alpha-1} and the standard properties of
  interpolation spaces gives \cref{eq:regu_alpha-3}.

  Finally, let us prove $ I_{0+}^\beta h(0) = 0 $ under the condition that $ h
  \in H_0^{1-\beta}(0,T) $ and $ \beta \neq 0.5 $. Assuming $ 0.5 < \beta < 1 $,
  by \cite[Theorem 11.3]{Lions;1972} we have
  \[
    \int_0^T s^{-2(1-\beta)} \snm{h(s)}^2 \, \mathrm{d}s \lesssim
    \nm{h}_{H^{1-\beta}(0,T)}^2,
  \]
  so that, using the Cauchy-Schwarz inequality yields that, for $ 0 < t < T $,
  \begin{align*}
    \snm{ I_{0+}^\beta h(t) }
    & \leqslant
    \frac1{\Gamma(\beta)} \int_0^t (t-s)^{\beta-1} \snm{h(s)} ds \\
    & \leqslant
    \frac{t^{1-\beta}}{ \Gamma(\beta) }
    \int_0^t (t-s)^{\beta-1} \snm{ s^{\beta-1} h(s) } ds \\
    & \lesssim
    t^\frac12 \nm{h}_{H^{1-\beta}(0,T)}.
  \end{align*}
  This implies $ I_{0+}^\beta h(0) = 0 $, and so let us assume $ 0 < \beta < 0.5
  $. Since $ H_0^{1-\beta}(0,T) $ is continuously embedded into $ C[0,T] $, a
  simple calculation gives
  \[
    \snm{ I_{0+}^\beta h(t) } \lesssim t^\beta \nm{h}_{H^{1-\beta}(0,T)},
  \]
  which also implies $ I_{0+}^\beta h(0) = 0 $. This completes the proof.
\end{proof}

\begin{lem}
  \label{lem:refer2}
  Suppose that $ v \in L^2(0,T) $. Then $ \D_{0+}^\frac\alpha2 v \in L^2(0,T) $
  if and only if $ v \in H^\frac\alpha2(0,T) $; and $ \D_{T-}^\frac\alpha2 v \in
  L^2(0,T) $ if and only if $ v \in H^\frac\alpha2(0,T) $. Moreover, if $ v \in
  H^\frac\alpha2(0,T) $, then
  \begin{align}
    \nm{ \D_{0+}^\frac\alpha2 v }_{L^2(0,T)} \sim
    \nm{v}_{H^\frac\alpha2(0,T)} \sim
    \nm{ \D_{T-}^\frac\alpha2 v }_{L^2(0,T)},
    \label{eq:refer2-1} \\
    \left(
      \D_{0+}^\frac\alpha2 v,
      \D_{T-}^\frac\alpha2 v
    \right)_{L^2(0,T)} =
    \cos \left( \frac{\alpha \pi}2 \right)
    \snm{v}_{ H^\frac \alpha2(\mathbb R) }^2
    \sim \nm{v}_{ H^\frac \alpha2(0,T) }^2.
    \label{eq:refer2-2}
  \end{align}
\end{lem}
\begin{proof}
  If $ \D_{0+}^\frac\alpha2 v \in L^2(0,T) $, then a simple calculation gives
  that
  \[
    v(t) = c t^{\frac\alpha2 - 1} +
    (I_{0+}^\frac\alpha2 \D_{0+}^\frac\alpha2 v)(t), \quad
    0 < t < T,
  \]
  where $ c $ is a constant. Since $ v \in L^2(0,T) $, the constant $ c $ is
  zero, and so
  \[
    v = I_{0+}^\frac\alpha2 \D_{0+}^\frac\alpha2 v
    \quad \text{ in $ (0,T) $. }
  \]
  From \cref{lem:regu_alpha} it follows $ v \in H^\frac\alpha2(0,T) $.
  Similarly, we can prove that $ \D_{T-}^\frac\alpha2v \in L^2(0,T) $ implies $
  v \in H^\frac\alpha2(0,T) $.

  Conversely, if $ v \in H^\frac\alpha2(0,T) $, then extending $ v $ to $
  \mathbb R \setminus (0,T) $ by zero gives
  \[
    \nm{v}_{H^\frac\alpha2(\mathbb R)} \lesssim
    \nm{v}_{H^\frac\alpha2(0,T)}.
  \]
  From \cref{lem:F} and the Fourier-Plancherel formula it follows
  \begin{align*}
    \nm{ \D_{0+}^\frac \alpha2 v }_{L^2(0,T)} \leqslant
    \nm{ \D_{+}^\frac \alpha2 v }_{ L^2(\mathbb R) } \lesssim
    \nm{v}_{ H^\frac \alpha2(0,T) }, \\
    \nm{ \D_{-}^\frac \alpha2 v }_{L^2(0,T)} \leqslant
    \nm{ \D_{-}^\frac \alpha2 v }_{ L^2(\mathbb R) } \lesssim
    \nm{v}_{ H^\frac \alpha2(0,T) }.
  \end{align*}

  Finally, for the proofs of \cref{eq:refer2-1,eq:refer2-2}, we refer the reader
  to \cite[Lemma 2.4 and Theorem 2.13]{Ervin;2006}. This completes the proof.
\end{proof}
\begin{rem}
  Note that \cite[Thoerem 2.13]{Ervin;2006} shows $ H^\frac\alpha2(0,T) =
  J_{\mathrm{L},0}^\frac\alpha2(0,T) $, where $
  J_{\mathrm{L},0}^\frac\alpha2(0,T) $ denotes the closure of $ \mathcal D(0,T)
  $ with respect to the following norm:
  \[
    \nm{v}_{ J_{L,0}^\frac \alpha2(0,T) } :=
    \left(
      \nm{v}_{L^2(0,T)}^2 +
      \nm{ \D_{0+}^\frac \alpha2 v }_{L^2(0,T)}^2
    \right)^\frac12,
    \quad \forall v \in \mathcal D(0,T).
  \]
  Evidently, this result does not guarantee that if $ v $, $
  \D_{0+}^\frac\alpha2 v \in L^2(0,T) $ then $ v \in H^\frac\alpha2(0,T) $.
\end{rem}
\begin{rem}
  By the definition of $ H^\frac\alpha2(0,T; L^2(\Omega)) $ and the above lemma,
  it is easy to verify that $ \partial_t^\frac\alpha2 v \in L^2(Q_T) $ and
  \[
    \nm{ \partial_t^\frac \alpha2 v }_{L^2(Q_T)} \sim
    \nm{v}_{ H^\frac \alpha2( 0,T; L^2(\Omega)) },
  \]
  for all $ v \in H^\frac\alpha2(0,T; L^2(\Omega)) $.
\end{rem}

\begin{lem}
  \label{lem:Dalpha}
  Let $ v \in H^\frac\alpha2(0,T) $. Then $ \D_{0+}^\alpha v \in L^2(0,T) $ if
  and only if $ v \in H^\alpha(0,T) $ with $ v(0) = 0 $. Moreover, if $
  \D_{0+}^\alpha v \in L^2(0,T) $, then
  \begin{equation}
    \label{eq:Dalpha-1}
    \nm{v}_{H^\alpha(0,T)} \sim \nm{\D_{0+}^\alpha v}_{L^2(0,T)},
  \end{equation}
  and
  \begin{equation}
    \label{eq:Dalpha-2}
    \left(
      \D_{0+}^\alpha v, \varphi
    \right)_{L^2(0,T)} =
    \left(
      \D_{0+}^\frac\alpha2 v,
      \D_{T-}^\frac\alpha2 \varphi
    \right)_{L^2(0,T)}
  \end{equation}
  for all $ \varphi \in H^\frac\alpha2(0,T) $.
\end{lem}
\begin{proof}
  Assuming $ \D_{0+}^\alpha v \in L^2(0,T) $, let us show that $ v \in
  H^\alpha(0,T) $ with $ v (0) = 0 $, and
  \begin{equation}
    \label{eq:lxy-11}
    \nm{v}_{H^\alpha(0,T)} \lesssim \nm{\D_{0+}^\alpha v}_{L^2(0,T)}.
  \end{equation}
  A straightforward computing gives that
  \[
    v(t) = ct^{\alpha-1} + (I_{0+}^\alpha \D_{0+}^\alpha v) (t),
    \quad 0 < t < T,
  \]
  for some $ c \in \mathbb R $. Noting that \cref{lem:regu_alpha} implies $
  I_{0+}^\alpha \D_{0+}^\alpha v \in H^\alpha(0,T) $, the constant $ c $ must be
  zero since $ v \in H^\frac\alpha2(0,T) $, hence $ v \in H^\alpha(0,T) $, and
  using \cref{lem:regu_alpha} again yields \cref{eq:lxy-11}. Besides, the
  Cauchy-Schwarz inequality implies
  \[
    \snm{v(t)} \leqslant \frac1{ \Gamma(\alpha) }
    \sqrt{ \frac{t^{2\alpha-1}}{2\alpha-1} }
    \nm{ \D_{0+}^\alpha v }_{L^2(0,t)},
    \quad 0 < t < T,
  \]
  so that $ v(0) = \lim_{t \to {0+}} v(t) = 0 $.

  Next, assuming $ v \in H^\alpha(0,T) $ with $ v(0) = 0 $, let us show that
  \begin{equation}
    \label{eq:lxy-12}
    \nm{\D_{0+}^\alpha v}_{L^2(0,T)} \lesssim \nm{v}_{H^\alpha(0,T)}.
  \end{equation}
  Since there exists an $ H^\alpha(\mathbb R) $-extension $ \widetilde v $
  of $ v $ satisfying
  \[
    \supp\widetilde v  \subset [0,2T] \text{ and }
    \nm{\widetilde v}_{H^\alpha(\mathbb R)} \sim
    \nm{v}_{H^\alpha(0,T)},
  \]
  using \cref{lem:F}, the Fourier-Plancherel formula, and the fact that $
  \D_{0+}^\alpha v = \D_{+}^\alpha \widetilde v $ in $ (0, T) $, we obtain
  \[
    \nm{\D_{0+}^\alpha v}_{L^2(0,T)} =
    \nm{\D_{+}^\alpha \widetilde v}_{L^2(0,T)} \leqslant
    \nm{\D_{+}^\alpha \widetilde v}_{L^2(\mathbb R)} =
    \snm{\widetilde v}_{H^\alpha(\mathbb R)}
    \lesssim
    \nm{v}_{H^\alpha(0,T)}.
  \]
  This proves the estimate \cref{eq:lxy-12}.

  Now we have proved the equivalence of $ \D_{0+}^\alpha v \in L^2(0,T) $ and $
  v \in H^\alpha(0,T) $ with $ v(0) = 0 $; moreover, combining
  \cref{eq:lxy-11,eq:lxy-12} proves \cref{eq:Dalpha-1}. Therefore, it remains to
  prove \cref{eq:Dalpha-2}, and since $ \mathcal D(0,T) $ is dense in $
  H^\frac\alpha2(0,T) $, by \cref{lem:refer2} it suffices to show that
  \cref{eq:Dalpha-2} holds for all $ \varphi \in \mathcal D(0,T) $. To this end,
  we argue as follows. Since $ v \in H^\frac\alpha2(0,T) $,
  \cref{lem:regu_alpha} implies $ I_{0+}^{1-\frac\alpha2} v \in H^1(0,T) $ with
  $ I_{0+}^{1-\frac\alpha2} v(0) = 0 $. Also, an elementary computing yields
  \[
    I_{T-}^{1-\alpha} \varphi' =
    -I_{T-}^{1-\frac\alpha2} \D \D I_{T-}^{1-\frac\alpha2} \varphi
    \quad \text{ for all $ \varphi \in \mathcal D(0,T) $. }
  \]
  Consequently, using integration by parts gives that
  \begin{align*}
    {} &
    \left(
      \D_{0+}^\alpha v, \varphi
    \right)_{L^2(0,T)} =
    -\left(
      I_{0+}^{1-\alpha} v, \varphi'
    \right)_{L^2(0,T)} =
    -\left(
      v, I_{T-}^{1-\alpha} \varphi'
    \right)_{L^2(0,T)} \\
    ={} &
    \left(
      v, I_{T-}^{1-\frac\alpha2} \D \D I_{T-}^{1-\frac\alpha2} \varphi
    \right)_{L^2(0,T)} =
    \left(
      I_{0+}^{1-\frac\alpha2} v, \D \D I_{T-}^{1-\frac\alpha2} \varphi
    \right)_{L^2(0,T)} \\
    ={} &
    -\left(
      \D I_{0+}^{1-\frac\alpha2} v,
      \D I_{T-}^{1-\frac\alpha2} \varphi
    \right)_{L^2(0,T)} =
    \left(
      \D_{0+}^\frac\alpha2 v,
      \D_{T-}^\frac\alpha2 \varphi
    \right)_{L^2(0,T)},
  \end{align*}
  for all $ \varphi \in \mathcal D(0,T) $. This proves \cref{eq:Dalpha-2}, and
  thus concludes the proof of the lemma.
\end{proof}
\begin{rem}
  Note that \cite[Lemma 2.6]{Li;Xu;2009} provides a proof of \cref{eq:Dalpha-2}
  under the condition that $ v \in H^1(0,T) $ with $ v(0) = 0 $.
\end{rem}



\section{Regularity of a fractional ordinary equation}
\label{sec:ode}
This section is devoted to investigating the regularity properties of the
following problem: seek $ y \in C[0,T] $ with $ y(0) = y_0 \in \mathbb R $ such
that
\begin{equation}
  \label{eq:ode}
  \D_{0+}^\alpha (y-y_0) + \lambda y = g
  \quad \text{ in $ (0, T) $, }
\end{equation}
where $ \lambda $ is a positive constant, and $ g \in L^2(0,T) $. Throughout
this section we assume that $ 1 \lesssim \lambda $, and, for the sake of
rigorousness, we also understand $ y_0 $ by a function with support $ [0,T] $.

The main results of this section are the following two theorems.
\begin{thm}
  \label{thm:ode-1}
  The problem \cref{eq:ode} has a unique solution $ y \in H^\alpha(0,T) $ with $
  y(0) = y_0 $ such that
  \begin{align}
    \nm{y}_{ H^\frac \alpha2(0,T) } +
    \lambda^\frac12 \nm{y}_{L^2(0,T)}
    &\lesssim
    \lambda^{-\frac12} \nm{g}_{L^2(0,T)} + \snm{y_0},
    \label{eq:ode-1-1} \\
    \nm{y}_{ H^\alpha(0,T) } +
    \lambda^\frac12 \nm{y}_{ H^\frac \alpha2(0,T) }
    & \lesssim
    \nm{g}_{L^2(0,T)} + \lambda^\frac12 \snm{y_0},
    \label{eq:ode-1-2} \\
    \lambda \nm{y}_{L^2(0,T)}
    & \lesssim
    \nm{g}_{L^2(0,T)} + \lambda^\frac12 \snm{y_0}.
    \label{eq:ode-1-3}
  \end{align}
  Moreover, if $ g \in H^{1-\alpha}(0,T) $ then
  \begin{equation}
    \label{eq:ode-1-4}
    \nm{y}_{H^1(0,T)} \lesssim \nm{g}_{ H^{1-\alpha}(0,T) } +
    \lambda \snm{y_0}.
  \end{equation}
\end{thm}

\begin{thm}
  \label{thm:ode-2}
  Suppose that $ g \in H^1(0,T) $. Then the solution $ y $ to problem
  \cref{eq:ode} satisfies the following estimates:
  \begin{align}
    \nm{y - S}_{ H^{1+\frac\alpha2}(0,T) } +
    \lambda^\frac12 \nm{y}_{H^1(0,T)}
    &\lesssim
    \lambda^{-\frac12} \nm{g}_{H^1(0,T)} +
    \snm{y_0} +
    \lambda^\frac12 \snm{ g(0) - \lambda y_0 },
    \label{eq:ode-2-1} \\
    \nm{y-S}_{ H^{1+\alpha}(0,T) } +
    \lambda^\frac12 \nm{y-S}_{ H^{1+\frac\alpha2}(0,T) }
    & \lesssim
    \nm{g}_{H^1(0,T)} + \snm{y_0} +
    \lambda \snm{ g(0) - \lambda y_0 },
    \label{eq:ode-2-2} \\
    \lambda \nm{y}_{H^1(0,T)}
    & \lesssim
    \nm{g}_{H^1(0,T)} + \lambda^\frac12 \snm{y_0} +
    \lambda \snm{ g(0) - \lambda y_0 },
    \label{eq:ode-2-3}
  \end{align}
  where
  \[
    S(t) := \frac{ g(0) - \lambda y_0 }{ \Gamma(1+\alpha) } t^\alpha,
    \quad 0 < t < T.
  \]
  Moreover, if $ 0.75 < \alpha < 1 $ and $ g \in H^{2-\alpha}(0,T) $, then
  \begin{equation}
    \label{eq:ode-2-4}
    \nm{y-S}_{H^2(0,T)}  \lesssim
    \nm{g}_{H^{2-\alpha}(0,T)} + \snm{y_0} +
    \lambda \snm{g(0) - \lambda y_0}.
  \end{equation}
\end{thm}

\noindent { \bf Proof of \cref{thm:ode-1} }
  Let us first show that problem \cref{eq:ode} has a unique solution $ y \in
  H^\alpha(0,T) $ with $ y(0) = y_0 $. By \cref{lem:refer2}, and the famous
  Riesz representation theorem, there exists a unique $ \widetilde y \in
  H^\frac\alpha2(0,T) $ such that
  \begin{equation}
    \label{eq:ode-1-10}
    \left(
      \D_{0+}^\frac\alpha2 \widetilde y,
      \D_{T-}^\frac\alpha2 z
    \right)_{L^2(0,T)} +
    \lambda \left( \widetilde y, z \right)_{L^2(0,T)} =
    \left( g, z \right)_{L^2(0,T)} -
    \lambda (y_0, z)_{L^2(0,T)}
  \end{equation}
  for all $ z \in H^\frac\alpha2(0,T) $. Since \cref{lem:regu_alpha} implies $
  I_{0+}^{1-\frac\alpha2} \widetilde y(0) = 0 $, using integration by parts
  gives that
  \begin{align*}
    {} &
    \left(
      \D_{0+}^\frac\alpha2 \widetilde y,
      \D_{T-}^\frac\alpha2 z
    \right)_{L^2(0,T)} =
    -\left(
      \D I_{0+}^{1-\frac\alpha2} \widetilde y,
      \D I_{T-}^{1-\frac\alpha2} \varphi
    \right)_{L^2(0,T)} \\
    ={} &
    \left(
      I_{0+}^{1-\frac\alpha2} \widetilde y,
      I_{T-}^{1-\frac\alpha2} \varphi''
    \right)_{L^2(0,T)} =
    \left(
      \widetilde y,
      I_{T-}^{2-\alpha} \varphi''
    \right)_{L^2(0,T)} \\
    ={} &
    -\left(
      \widetilde y, I_{T-}^{1-\alpha} \varphi'
    \right)_{L^2(0,T)} =
    -\left(
      I_{0+}^{1-\alpha} \widetilde y,
      \varphi'
    \right)_{L^2(0,T)} \\
    ={} &
    \dual{
      \D_{0+}^\alpha \widetilde y, \varphi
    } \\
  \end{align*}
  for all $ \varphi \in \mathcal D(0,T) $, where $ \dual{\cdot,\cdot} $ denotes
  the duality pairing between $ \mathcal D'(\mathbb R) $ and $ \mathcal
  D(\mathbb R) $. From \cref{eq:ode-1-10} it follows that
  \[
    \D_{0+}^\alpha \widetilde y = g - \lambda (y_0 + \widetilde y)
    \quad \text{ in $ (0,T) $, }
  \]
  and then, \cref{lem:Dalpha} implies $ \widetilde y \in H^\alpha(0,T) $ with $
  \widetilde y(0) = 0 $. Obviously, $ y:= y_0 + \widetilde y $ in $ (0,T) $ is
  an $ H^\alpha(0,T) $-solution of \cref{eq:ode}, and, by
  \cref{lem:Dalpha,lem:refer2}, it is clear that this $ H^\alpha(0,T) $-solution
  is unique.

  Then let us show \cref{eq:ode-1-1}. Multiplying both sides of \cref{eq:ode} by
  $ y $, and integrating in $ (0,T) $, we obtain
  \[
    \left(
      \D_{0+}^\alpha(y-y_0), y
    \right)_{L^2(0,T)} +
    \lambda \nm{y}_{L^2(0,T)} =
    \left( g, y \right)_{L^2(0,T)},
  \]
  so that
  \[
    \left(
      \D_{0+}^\alpha(y-y_0), y-y_0
    \right)_{L^2(0,T)} +
    \lambda \nm{y}_{L^2(0,T)} =
    \left( g, y \right)_{L^2(0,T)} -
    \left(
      \D_{0+}^\alpha (y-y_0), y_0
    \right)_{L^2(0,T)}.
  \]
  From \cref{lem:Dalpha} it follows
  \[
    \left(
      \D_{0+}^\frac\alpha2(y-y_0), \D_{T-}^\frac\alpha2 (y-y_0)
    \right)_{L^2(0,T)} +
    \lambda \nm{y}_{L^2(0,T)} =
    \left( g, y \right)_{L^2(0,T)} -
    \left(
      \D_{0+}^\frac\alpha2 (y-y_0), \D_{T-}^\frac\alpha2 y_0
    \right)_{L^2(0,T)}.
  \]
  Therefore, using \cref{lem:refer2} and the Cauchy's inequality with $ \epsilon
  $ gives
  \[
    \nm{y-y_0}_{ H^\frac \alpha2(0,T) }^2 +
    \lambda \nm{y}_{L^2(0,T)}^2 \lesssim
    \lambda^{-1} \nm{g}_{L^2(0,T)}^2 + y_0^2,
  \]
  which, together with the estimate
  \[
    \nm{y_0}_{H^\frac\alpha2(0,T)} \lesssim  \snm{y_0},
  \]
  yields \cref{eq:ode-1-1}.

  Next, let us show \cref{eq:ode-1-2}. Multiplying both sides of \cref{eq:ode}
  by $ \D_{0+}^\alpha(y-y_0) $, and integrating in $ (0,T) $ yield
  \[
    \nm{ \D_{0+}^\alpha(y-y_0) }_{L^2(0,T)}^2 +
    \lambda \left(
      y, \D_{0+}^\alpha( y-y_0 )
    \right)_{L^2(0,T)} =
    \left(
      g, \D_{0+}^\alpha( y-y_0 )
    \right)_{L^2(0,T)},
  \]
  so that
  \[
    \nm{ \D_{0+}^\alpha(y-y_0) }_{L^2(0,T)}^2 +
    \lambda \left(
      y-y_0, \D_{0+}^\alpha( y-y_0 )
    \right)_{L^2(0,T)} =
    \left(
      g, \D_{0+}^\alpha( y-y_0 )
    \right)_{L^2(0,T)} -
    \lambda \left(
      y_0, \D_{0+}^\alpha ( y-y_0 )
    \right)_{L^2(0,T)}.
  \]
  From \cref{lem:refer2,lem:Dalpha} it follows
  \[
    \nm{y-y_0}_{H^\alpha(0,T)}^2 +
    \lambda \nm{ y-y_0 }_{H^\frac\alpha2(0,T)}^2
    \lesssim
    \nm{g}_{L^2(0,T)} \nm{y-y_0}_{H^\alpha(0,T)} +
    \lambda \nm{y-y_0}_{H^\frac\alpha2(0,T)} \snm{y_0},
  \]
  hence the Cauchy's inequality with $ \epsilon $ implies
  \begin{equation}
    \label{eq:ode-1-2-1}
    \nm{y-y_0}_{H^\alpha(0,T)}^2 +
    \lambda \nm{y-y_0}_{ H^\frac \alpha2(0,T) }^2
    \lesssim \nm{g}_{L^2(0,T)}^2 + \lambda y_0^2.
  \end{equation}
  Therefore, \cref{eq:ode-1-2} follows from the following estimates:
  \begin{align*}
    \nm{y}_{ H^\alpha(0,T) } & \lesssim
    \nm{y-y_0}_{ H^\alpha(0,T) } + \snm{y_0}, \\
    \nm{y}_{ H^\frac \alpha2(0,T) } & \lesssim
    \nm{y-y_0}_{ H^\frac \alpha2(0,T) } + \snm{y_0}.
  \end{align*}

  Now, let us show \cref{eq:ode-1-3}. Since \cref{eq:ode} implies
  \[
    \lambda y = g - \D_{0+}^\alpha (y-y_0)
    \quad \text{ in $ (0,T) $, }
  \]
  by \cref{lem:Dalpha} we obtain
  \[
    \lambda^2 \nm{y}_{L^2(0,T)}^2 \lesssim
    \nm{g}_{L^2(0,T)}^2 + \nm{ y-y_0 }_{H^\alpha(0,T)}^2.
  \]
  Then \cref{eq:ode-1-3} is a direct conclusion of \cref{eq:ode-1-2-1}.

  Finally, let us show \cref{eq:ode-1-4}. By \cref{eq:ode} a simple computing
  gives
  \[
    y = y_0 + I_{0+}^\alpha (g - \lambda y)
    \quad \text{ in $ (0, T) $, }
  \]
  Since $ y \in H^\alpha(0,T) \subset H^{1-\alpha}(0,T) $, from
  \cref{lem:regu_alpha} it follows $ y \in H^1(0,T) $, and so
  \[
    \D y = \D I_{0+}^\alpha (g - \lambda y)
    \quad \text{ in $ (0,T) $. }
  \]
  which implies
  \[
    \D y + \lambda \D I_{0+}^\alpha(y-y_0) =
    \D I_{0+}^\alpha g - \D I_{0+}^\alpha y_0
    \quad \text{ in $ (0,T) $. }
  \]
  Since $ y \in H^1(0,T) $ with $ y(0) = y_0 $ implies
  \[
    \D I_{0+}^\alpha (y-y_0) = I_{0+}^\alpha \D y
    \quad \text{ in $ (0,T) $, }
  \]
  it follows
  \[
    \D y + \lambda I_{0+}^\alpha \D y = \D I_{0+}^\alpha g -
    \lambda \D I_{0+}^\alpha y_0
    \quad \text{ in $ (0,T) $. }
  \]
  Multiplying both sides of the above equation by $ \D y $, and integrating in $
  (0,T) $, we obtain
  \[
    \nm{\D y}_{L^2(0,T)}^2 +
    \lambda \left( I_{0+}^\alpha \D y, \D y \right)_{L^2(0,T)} =
    \left( \D I_{0+}^\alpha g, \D y \right)_{L^2(0,T)} -
    \lambda \left( \D I_{0+}^\alpha y_0, \D y \right)_{L^2(0,T)},
  \]
  so that, using \cref{lem:regu_alpha} and the Cauchy's inequality with $
  \epsilon $ gives
  \[
    \nm{\D y}_{L^2(0,T)}^2 +
    \lambda \left( I_{0+}^\alpha \D y, \D y \right)_{L^2(0,T)} \lesssim
    \nm{g}_{ H^{1-\alpha}(0,T) }^2 + \lambda^2 y_0^2.
  \]
  Since \cref{eq:ode-1-1} implies
  \[
    \nm{y}_{L^2(0,T)} \lesssim \lambda^{-1} \nm{g}_{L^2(0,T)} +
    \lambda^{-\frac12} \snm{y_0},
  \]
  to prove \cref{eq:ode-1-4} it suffices to to show that
  \begin{equation}
    \label{eq:ode-1-4-1}
    \left( I_{0+}^\alpha \D y, \D  y \right)_{L^2(0,T)}
    \geqslant 0.
  \end{equation}
  To this end, let us define
  \[
    v(t) :=
    \begin{cases}
      \D y(t) & \text{ if $ 0 < t < T $, } \\
      0       & \text{ otherwise. }
    \end{cases}
  \]
  Since $ \frac\alpha2 < \frac12 $, it is easy to verify that $
  I_{0+}^\frac\alpha2 v$, $ I_{T-}^\frac\alpha2 v \in L^2(\mathbb R) $, and
  then, using \cref{lem:F} and the famous Parseval's theorem yields
  \[
    \left(
      I_{0+}^\alpha v, v
    \right)_{L^2(0,T)} =
    \left(
      I_{0+}^\frac\alpha2 v, I_{T-}^\frac\alpha2 v
    \right)_{L^2(0,T)} =
    \left(
      I_{0+}^\frac\alpha2 v, I_{T-}^\frac\alpha2 v
    \right)_{L^2(\mathbb R)} =
    \cos \left( \frac{\alpha\pi}2 \right)
    \int_{\mathbb R} \snm{w}^{-\alpha} \snm{\mathcal F v(\xi)}^2
    \, \mathrm{d}w \geqslant 0,
  \]
  which proves \cref{eq:ode-1-4-1}. This concludes the proof of
  \cref{eq:ode-1-4}, and thus the proof of \cref{thm:ode-1}.
\hfill\ensuremath{\blacksquare}

\noindent
{ \bf Proof of \cref{thm:ode-2}. }
  By \cref{thm:ode-1} we see that there exists uniquely $ z \in H^\alpha(0,T) $
  with $ z(0) = 0 $ such that
  \[
    \D_{0+}^\alpha z + \lambda z = \D g - \lambda S'
    \quad \text{ in $ (0,T) $. }
  \]
  Moreover,
  \begin{align*}
    \nm{z}_{ H^\frac \alpha2(0,T) } +
    \lambda^\frac12 \nm{z}_{L^2(0,T)}
    &\lesssim
    \lambda^{-\frac12} \nm{\D g}_{L^2(0,T)} +
    \lambda^\frac12 \snm{ g(0) - \lambda y_0 }, \\
    \nm{z}_{ H^\alpha(0,T) } +
    \lambda^\frac12 \nm{z}_{ H^\frac \alpha2(0,T) }
    & \lesssim
    \nm{\D g}_{L^2(0,T)} +
    \lambda \snm{ g(0) - \lambda y_0 }, \\
    \lambda \nm{z}_{L^2(0,T)}
    & \lesssim
    \nm{\D g}_{L^2(0,T)} +
    \lambda \snm{ g(0) - \lambda y_0 },
  \end{align*}
  and if $ 0.75 < \alpha < 1 $ and $ g \in H^{2-\alpha}(0,T) $, then
  \[
    \nm{z}_{H^1(0,T)} \lesssim
    \nm{\D g}_{H^{1-\alpha}(0,T)} + \lambda \snm{g(0) - \lambda y_0},
  \]
  since $ \D g - \lambda S' \in H^{1-\alpha}(0,T) $. Putting
  \begin{equation}
    \label{eq:100}
    y := y_0 + S + I_{0+}z  \quad \text{ in $ (0,T) $, }
  \end{equation}
  if $ y $ is the solution of problem \cref{eq:ode}, then, by \cref{thm:ode-1}
  and the estimate
  \[
    \nm{S}_{H^1(0,T)} \lesssim \snm{g(0) - \lambda y_0},
  \]
  some simple manipulation yields
  \cref{eq:ode-2-1,eq:ode-2-2,eq:ode-2-3,eq:ode-2-4}. Therefore, to complete the
  proof of this theorem, it remains to show that \cref{eq:100} is the solution
  of problem \cref{eq:ode}.

  To do so, note that the definition of $ z $ implies
  \[
    I_{0+} \D_{0+}^\alpha z + \lambda I_{0+}z = g - g(0) - \lambda S
    \quad \text{ in $ (0, T) $. }
  \]
  Also, from the fact $ z \in H^\alpha(0,T) $ with $ z(0) = 0 $ it follows
  \[
    I_{0+} \D_{0+}^\alpha z = \D_{0+}^\alpha I_{0+} z
    \quad \text{ in $ (0, T) $. }
  \]
  Consequently,
  \[
    \D_{0+}^\alpha I_{0+}z + \lambda I_{0+}z = g - g(0) - \lambda S
    \quad \text{ in $ (0,T) $, }
  \]
  and so
  \[
    \D_{0+}^\alpha (I_{0+}z + S) + \lambda (I_{0+}z + S + y_0) -
    \D_{0+}^\alpha S + g(0) - \lambda y_0 = g
    \quad \text{ in $ (0,T) $. }
  \]
  Since a direct computing gives
  \[
    -\D_{0+}^\alpha S + g(0) - \lambda y_0 = 0
    \quad \text{ in $ (0,T) $, }
  \]
  it is evident that \cref{eq:100} is the solution of \cref{eq:ode} indeed. This
  completes the proof of \cref{thm:ode-2}.
\hfill\ensuremath{\blacksquare}

\section{Main results}
\label{sec:main}
In this section we shall employ the results developed in \cref{sec:ode} to
analyze the regularity properties of problem \cref{eq:model}. Let us start by
introducing some notation and conventions. For each $ v \in L^2(Q_T) $, we can
naturally regard it as an $ L^2(\Omega) $-valued function with domain $ (0,T) $,
an element of $ L^2(0,T; L^2(\Omega)) $, and, for convenience, we also use $ v $
to denote this $ L^2(\Omega) $-valued function. It is well known that, there
exists, in $ H_0^1(\Omega) \cap H^2(\Omega) $, an orthonormal basis $ \{ \phi_k
|\ k \in \mathbb N \} $ of $ L^2(\Omega) $, and a nondecreasing sequence $ \{
\lambda_k > 0 |\ k \in \mathbb N \} $ such that
\[
  -\Delta \phi_k = \lambda_k \phi_k
  \quad \text{ in $ \Omega $, for all $ k \in \mathbb N $. }
\]
Meanwhile, $ \{ \lambda_k^{-1/2} \phi_k|\ k \in \mathbb N \} $ is an orthonormal
basis of $ H_0^1(\Omega) $ equipped with the inner product $ (\nabla\cdot,
\nabla\cdot)_{L^2(\Omega)} $. For each $ k \in \mathbb N $, define $ c_k \in
C[0,T] $ with $ c_k(0) := (u_0, \phi)_{L^2(\Omega)} $ by
\[
  -\D_{0+}^\alpha \big( c_k-c_k(0) \big) + \lambda_k c_k = f_k
  \quad \text{ in $ (0,T) $, }
\]
where
\[
  f_k(t) := \int_\Omega f(x,t) \phi_k(x) \, \mathrm{d}x,
  \quad 0 < t < T.
\]
Finally, define
\begin{equation}
  \label{eq:u}
  u(t) := \sum_{k=0}^\infty c_k(t) \phi_k,
  \quad 0 < t < T.
\end{equation}

By \cref{thm:ode-1,thm:ode-2}, we readily obtain the following estimates on the
above $ u $.
\begin{thm}
  \label{thm:esti-u}
  The $ u $ defined by \cref{eq:u} satisfies the following estimates:
  \begin{itemize}
    \item If $ f \in L^2(0,T; L^2(\Omega)) $ and $ u_0 \in H_0^1(\Omega) $, then
      \begin{equation}
        \label{eq:esti-u-1}
        \nm{u}_{H^\alpha(0,T; L^2(\Omega))} +
        \nm{u}_{H^\frac\alpha2(0,T; H_0^1(\Omega))} +
        \nm{u}_{L^2(0,T; H^2(\Omega))}
        \lesssim
        \nm{f}_{L^2(0,T; L^2(\Omega))} +
        \nm{u_0}_{H_0^1(\Omega)}.
      \end{equation}
    \item If $ f \in H^{1-\alpha}(0,T; L^2(\Omega)) $ and $ u_0 \in
      H_0^1(\Omega) \cap H^2(\Omega) $, then
      \begin{equation}
        \label{eq:esti-u-2}
        \nm{u}_{H^1(0,T; L^2(\Omega))} \lesssim
        \nm{f}_{H^{1-\alpha}(0,T; L^2(\Omega))} +
        \nm{u_0}_{H^2(\Omega)}.
      \end{equation}
    \item If $ f \in H^1(0,T; L^2(\Omega)) $ with $ f(0) \in H^2(\Omega) $ and $
      u_0 \in H_0^1(\Omega) \cap H^4(\Omega) $, then
      \begin{align}
        {} &
        \nm{u-S}_{H^{1+\alpha}(0,T; L^2(\Omega))} +
        \nm{u-S}_{H^{1+\frac\alpha2}(0,T; H_0^1(\Omega))} +
        \nm{u}_{H^1(0,T;H^2(\Omega))} \notag \\
        \lesssim{} &
        \nm{f}_{H^1(0,T; L^2(\Omega))} + \nm{f(0)}_{H^2(\Omega)} +
        \nm{u_0}_{H^4(\Omega)}. \label{eq:esti-u-3}
      \end{align}
    \item If $ 0.75 < \alpha < 1 $, $ f \in H^{2-\alpha}(0,T; L^2(\Omega)) $
      with $ f(0) \in H^2(\Omega) $, and $ u_0 \in H_0^1(\Omega) \cap
      H^4(\Omega) $, then
      \begin{equation}
        \label{eq:esti-u-4}
        \nm{u-S}_{H^2(0,T; L^2(\Omega))} \lesssim
        \nm{f}_{H^{2-\alpha}(0,T; L^2(\Omega))} +
        \nm{f(0)}_{H^2(\Omega)} + \nm{u_0}_{H^4(\Omega)}.
      \end{equation}
  \end{itemize}
  Above,
  \[
    S(t) := \sum_{k=0}^\infty \frac{ f_k(0) - \lambda_k c_k(0) }
    { \Gamma(1+\alpha) } t^\alpha \phi_k,
    \quad 0 < t < T.
  \]
\end{thm}
\begin{proof}
  Since by \cref{thm:ode-1,thm:ode-2}, the proofs of
  \cref{eq:esti-u-1,eq:esti-u-2,eq:esti-u-3,eq:esti-u-4} are straightforward,
  below we only prove \cref{eq:esti-u-1}. To this end, note that $ f \in
  L^2(0,T; L^2(\Omega)) $ and $ u_0 \in H_0^1(\Omega) $ imply
  \begin{align*}
    \sum_{k=0}^\infty \lambda_k c_k(0)^2 &=
    \nm{u_0}_{H_0^1(\Omega)}^2, \\
    \sum_{k=0}^\infty \nm{f_k}_{L^2(0,T)}^2 &=
    \nm{f}_{ L^2( 0,T; L^2(\Omega) ) }^2.
  \end{align*}
  From \cref{thm:ode-1} it follows
  \begin{equation}
    \label{eq:lxy-1}
    \sum_{k=0}^\infty \left(
      \nm{c_k}_{H^\alpha(0,T)}^2 +
      \lambda_k \nm{c_k}_{H^\frac\alpha2(0,T)}^2 +
      \lambda_k^2 \nm{c_k}_{L^2(0,T)}^2
    \right) \lesssim
    \nm{f}_{L^2(0,T; L^2(\Omega))}^2 +
    \nm{u_0}_{H_0^1(\Omega)}^2.
  \end{equation}
  Obviously, the above estimate implies
  \[
    \nm{u}_{ H^\alpha ( 0,T; L^2(\Omega) ) } +
    \nm{u}_{ H^\frac \alpha2( 0,T; H_0^1(\Omega) ) } \lesssim
    \nm{f}_{L^2(0,T; L^2(\Omega))} +
    \nm{u_0}_{H_0^1(\Omega)},
  \]
  and therefore, it remains to prove that
  \begin{equation}
    \label{eq:lxy-2}
    \nm{u}_{L^2(0,T; H^2(\Omega))} \lesssim
    \nm{f}_{L^2(0,T; L^2(\Omega))} +
    \nm{u_0}_{H_0^1(\Omega)}.
  \end{equation}
  To do so, using the standard estimate that $ \nm{v}_{H^2(\Omega)} \lesssim
  \nm{\Delta v}_{L^2(\Omega)} $ for all $ v \in H_0^1(\Omega) $ such that $
  \Delta v \in L^2(\Omega) $, we obtain
  \begin{align*}
    {} &
    \nm{u}_{L^2(0,T; H^2(\Omega))} =
    \left(
      \int_0^T \nm{u(t)}_{H^2(\Omega)}^2 \, \mathrm{d}t
    \right)^\frac12 \lesssim
    \left(
      \nm{\Delta u(t)}_{L^2(\Omega)}^2 \, \mathrm{d}t
    \right)^\frac12 \\
    ={} &
    \left(
      \int_0^T \sum_{k=0}^\infty \lambda_k^2 c_k(t)^2 \, \mathrm{d}t
    \right)^\frac12 =
    \left(
      \sum_{k=0}^\infty \lambda_k^2 \nm{c_k}_{L^2(0,T)}^2
    \right)^\frac12,
  \end{align*}
  which, together with \cref{eq:lxy-1}, yields \cref{eq:lxy-2}. This prove the
  estimate \cref{eq:esti-u-1}, and concludes the proof of the theorem.
\end{proof}

Finally, let us show in what sense $ u $, given by \cref{eq:u}, is a solution to
problem \cref{eq:model}.
\begin{thm}
  \label{thm:u}
  Suppose that $ f \in L^2(0,T; L^2(\Omega)) $ and $ u_0 \in H_0^1(\Omega) $.
  Then $ u $, defined by \cref{eq:u}, satisfies that $ u \in C([0,T];
  L^2(\Omega)) $ with $ u(0) = u_0 $, and that
  \begin{equation}
    \label{eq:weak-sol}
    \left(
      \partial_t^\alpha (u-u_0),
      \varphi
    \right)_{L^2(Q_T)} +
    \left( \nabla u, \nabla \varphi \right)_{L^2(Q_T)} =
    \left( f, \varphi \right)_{L^2(Q_T)},
  \end{equation}
  for all $ \varphi \in L^2(0,T; H_0^1(\Omega)) $.
\end{thm}
\begin{proof}
  By \cref{thm:esti-u} we have $ u \in H^\alpha(0,T; L^2(\Omega)) $, and as $
  H^\alpha(0,T) $ is continuously embedded into $ C^{\alpha-0.5}[0,T] $,
  modifying $ u $ on a set of measure zero yields
  \begin{align*}
    {} &
    \nm{ u(t+h) - u(t) }_{L^2(\Omega)}^2 =
    \sum_{k=0}^\infty \snm{c_k(t+h) - c_k(t)}^2 \\
    \lesssim{} &
    \sum_{k=0}^\infty \snm{h}^{2\alpha-1}
    \nm{c_k}_{H^\alpha(0,T)}^2 =
    \snm{h}^{2\alpha-1} \nm{u}_{H^\alpha(0,T; L^2(\Omega))}^2,
  \end{align*}
  for all $ 0 \leqslant t \leqslant T $ and $ h $ such that $ 0 \leqslant t+h
  \leqslant T $. This implies $ u \in C([0,T]; L^2(\Omega)) $, and by the
  definition \cref{eq:u} of $ u $, it is obvious that $ u(0) = u_0 $.

  Since $ u \in H^\alpha(0,T; L^2(\Omega)) $, \cref{lem:Dalpha} implies
  \[
    \sum_{k=0}^\infty \D_{0+}^\alpha (c_k - c_k(0)) \phi_k \in L^2(Q_T),
  \]
  and it is easy to verify that
  \[
    \partial_t^\alpha(u - u_0) =
    \sum_{k=0}^\infty \D_{0+}^\alpha(c_k - c_k(0)) \phi_k
    \quad \text{ in $ L^2(Q_T) $. }
  \]
  Thus, the definitions of $ c_k $'s indicate that
  \[
    \left(
      \partial_t^\alpha(u-u_0), \eta \phi_k
    \right)_{L^2(Q_T)} +
    \left(
      \nabla u, \nabla \phi_k \eta
    \right)_{L^2(Q_T)} =
    \left(
      f, \eta\phi_k
    \right)_{L^2(Q_T)}
  \]
  for all $ \eta \in \mathcal D(0,T) $ and $ k \in \mathbb N $, and hence
  \cref{eq:weak-sol} follows from the density of
  \[
    \text{span}\Big\{
      \eta \phi_k | \
      \eta \in \mathcal D(0,T),\ k \in \mathbb N
    \Big\}
  \]
  in $ L^2(0,T; H_0^1(\Omega)) $. This
  proves the theorem.
\end{proof}
\begin{rem}
  Under the same condition as \cref{thm:u}, suppose $ \widetilde u \in
  H^\alpha(0,T; L^2(\Omega)) \cap L^2(0,T; H_0^1(\Omega)) $ with $ \widetilde
  u(0) = u_0 $ also satisfies that
  \[
    \left(
      \partial_t^\alpha (\widetilde u-u_0),
      \varphi
    \right)_{L^2(Q_T)} +
    \left( \nabla \widetilde u, \nabla \varphi \right)_{L^2(Q_T)} =
    \left( f, \varphi \right)_{L^2(Q_T)},
  \]
  for all $ \varphi \in L^2(0,T; H_0^1(\Omega)) $. Putting
  \[
    w(t) := u(t) - \widetilde u(t), \quad 0 < t < T,
  \]
  we have
  \[
    \left( \partial_t^\alpha w, w \right)_{L^2(Q_T)} +
    \nm{w}_{ L^2(0,T; H_0^1(\Omega)) } = 0.
  \]
  Since $ w \in H^\alpha(0,T; L^2(\Omega)) $ with $ w(0) = 0 $, by
  \cref{lem:Dalpha} we obtain
  \[
    \left( \partial_t^\alpha w, w \right)_{L^2(Q_T)} \sim
    \nm{w}_{ H^\frac\alpha2(0,T; L^2(\Omega)) }^2,
  \]
  so that
  \[
    \nm{w}_{ H^\frac\alpha2(0,T; L^2(\Omega)) } =
    \nm{w}_{ L^2(0,T; H_0^1(\Omega)) } = 0.
  \]
  This implies $ w = 0 $ and hence $ u = \widetilde u $. Therefore, if we call $
  u \in H^\alpha(0,T; L^2(\Omega)) \cap L^2(0,T; H_0^1(\Omega)) $ a weak
  solution of \cref{eq:model} such that $ u(0) = u_0 $, and \cref{eq:weak-sol}
  holds for all $ \varphi \in L^2(0,T; H_0^1(\Omega)) $, then \cref{eq:model}
  has a unique weak solution given by \cref{eq:u}.
\end{rem}

\begin{rem}
  By \cref{thm:esti-u}, it is expected that taking $ t^\alpha $ as a basis
  function in time may improve the algorithms developed in
  \cite{Li;Xu;2009,Zheng;2015} noticeably.
\end{rem}

\appendix
\section{Proof of \cref{lem:F}}
\label{sec:app}
Suppose that $ 0 < \beta < 1 $ and $ v \in L^1(\mathbb R) $ with support $ [0,1]
$. It is easy to verify that $ I_{+}^\beta v \in \mathcal S'(\mathbb R) $, and
so, $ \D_{+}^{1-\beta} v = \D I_{+}^\beta v \in \mathcal S'(\mathbb R) $;
therefore, both $ \mathcal F(I_{+}^\beta v) $ and $ \mathcal F(\D_{+}^{1-\beta}
v) $ make sense. Below we give an elementary proof of the following equalities:
\begin{align}
  \mathcal F( I_{+}^\beta v )(\xi) &=
  ( \mathrm{i}\xi )^{-\beta} \mathcal F v(\xi),
  \label{eq:app-1} \\
  \mathcal F( \D_{+}^{1-\beta} v )(\xi) &=
  ( \mathrm{i}\xi )^{1-\beta} \mathcal F v(\xi),
  \label{eq:app-2}
\end{align}
for all $ -\infty < \xi < \infty $.

Let us first introduce a function $ G: \mathbb R \to \mathbb C $ by
\[
  G(x) := \frac 1 { \Gamma(\beta) }
  \int_{\gamma_x} \E^{-z} z^{\beta-1} \, \mathrm{d}z,
  \quad -\infty < x < \infty,
\]
where $ \gamma_x $ denotes the directed smooth curve $ z(t) := \mathrm{i}tx,\ 0
< t < 1 $, in the complex plane. It is easy to see that $ G $ is continuous, and
by a trivial modification of Jordan's Lemma in complex analysis we obtain
\[
  \lim_{x \to \infty} G(x) = \lim_{x \to -\infty} G(x) =
  \frac1{\Gamma(\beta)} \int_0^\infty \E^{-t} t^{\beta-1} \, \mathrm{d}t = 1.
\]
Therefore, $ \sup_{x \in \mathbb R} \snm{G(x)} < \infty $.

Given $ \varphi \in \mathcal S(\mathbb R) $, since $ \mathcal F \varphi \in
\mathcal S(\mathbb R) $, it is obvious that
\[
  \int_0^\infty \snm{\mathcal F\varphi(x)}
  \int_0^x (x-t)^{\beta-1} \snm{v(t)} \, \mathrm{d}t \, \mathrm{d}x
  < \infty,
\]
and so
\begin{align*}
  {} &
  \int_{\mathbb R} I_{+}^\beta v(x) \mathcal F\varphi(x)  \, \mathrm{d}x \\
  ={} &
  \frac1{\Gamma(\beta)} \int_0^\infty \mathcal F\varphi(x) \int_0^x(x-t)^{\beta-1}
  v(t) \, \mathrm{d}t \, \mathrm{d}x \\
  ={} &
  \lim_{n \to \infty} \frac1{\Gamma(\beta)} \int_0^n \mathcal F \varphi(x)
  \int_0^x (x-t)^{\beta-1} v(t) \, \mathrm{d}t \, \mathrm{d}x,
\end{align*}
by the Lebesgue's dominated convergence theorem. For $ n \in \mathbb N_{>0} $,
since
\[
  \mathcal F\varphi(x) = \frac1{\sqrt{2\pi}}
  \int_{\mathbb R} \E^{-\mathrm{i}x\xi} \varphi(\xi) \, \mathrm{d}\xi,
  \quad -\infty < x < \infty,
\]
a straightforward computing gives
\[
  \frac1{\Gamma(\beta)} \int_0^n \mathcal F \varphi(x)
  \int_0^x (x-t)^{\beta-1} v(t) \, \mathrm{d}t =
  \int_{\mathbb R} h_n(\xi) \, \mathrm{d}\xi,
\]
where
\[
  h_n(\xi) =
  ( \mathrm{i} \xi )^{-\beta} \varphi(\xi) \times
  \frac1{\sqrt{2\pi}} \int_0^1
  \E^{ -\mathrm{i}t\xi }\ v(t)\ G\big( (n-t)\xi \big)
  \, \mathrm{d}t.
\]
Therefore,
\[
  \int_{\mathbb R}
  I_{+}^\beta v(x) \mathcal F\varphi(x)
  \, \mathrm{d}x =
  \lim_{n \to \infty} \int_{\mathbb R} h_n(\xi) \, \mathrm{d}\xi.
\]
Putting
\[
  h(\xi) := (2\pi)^{-\frac 12} \snm{\xi}^{-\beta}
  \snm{ \varphi(\xi) } \nm{v}_{L^1(0,1)}
  \nm{G}_{ L^\infty(\mathbb R) },
  \quad \xi \in \mathbb R \setminus \{ 0 \},
\]
we have
\[
  \snm{h_n(\xi)} \leqslant h(\xi) \quad
  \text{ for all $ \xi \in \mathbb R \setminus \{0\} $, }
\]
and it is clear that $ h \in L^1(\mathbb R) $ since $ \varphi \in \mathcal
S(\mathbb R) $. Also, given $ \xi \in \mathbb R \setminus \{0\} $, since $
G((n-t)\xi) \to 1 $ uniformly for all $ t \in [0,1] $ as $ n $ tends to
infinity, we deduce that

\[
  h_n(\xi) \to (\mathrm{i}\xi)^{-\beta} \varphi(\xi) \mathcal F v(\xi)
  \quad \text{ as $ n \to \infty $, }
\]
As a consequence, the Lebesgue's dominated convergence theorem implies
\[
  \lim_{n \to \infty} \int_{\mathbb R} h_n(\xi) \, \mathrm{d}\xi =
  \int_{\mathbb R} (\mathrm{i}\xi)^{-\beta} \mathcal Fv(\xi) \varphi(\xi) \,
  \mathrm{d}\xi,
\]
therefore,
\[
  \int_{\mathbb R} I_{0+}^\beta v(x) \mathcal F \varphi(x) \, \mathrm{d}x =
  \int_{\mathbb R} (\mathrm{i}\xi)^{-\beta} \mathcal Fv(\xi) \varphi(\xi) \,
  \mathrm{d}\xi.
\]
Since $ \varphi \in \mathcal S(\mathbb R) $ is arbitrary, this proves
\cref{eq:app-1}, and then \cref{eq:app-2} follows from the following
standard result:
\[
  \mathcal F(\D_{+} I_{+}^\beta v)(\xi) =
  \mathrm{i} \xi \mathcal F(I_{+}^\beta v)(\xi),
  \quad -\infty < \xi < \infty.
\]
\begin{rem}
  A trivial modification of the above analysis yields that
  \begin{align*}
    \mathcal F( I_{-}^\beta v )(\xi) &=
    ( -\mathrm{i} \xi )^{-\beta} \mathcal Fv (\xi), \\
    \mathcal F( \D_{-}^{1-\beta} v )(\xi) &=
    ( -\mathrm{i} \xi )^{1-\beta} \mathcal Fv (\xi),
  \end{align*}
  for all $ -\infty < \xi < \infty $.
\end{rem}


\end{document}